\documentclass[a4paper]{amsart}

\usepackage{amssymb}

\theoremstyle{plain}
\newtheorem{thm}{Theorem}[section]
\newtheorem{prop}[thm]{Proposition}
\newtheorem{lem}[thm]{Lemma}
\newtheorem{cor}[thm]{Corollary}
\newtheorem{conj}[thm]{Conjecture}

\theoremstyle{definition} 
\newtheorem{defn}[thm]{Definition}
\newtheorem{rem}[thm]{Remark}

\newtheorem*{ack}{Acknowledgments} 

\begin{document}



\title[minimal log discrepancies in positive characteristic]{minimal log discrepancies in positive characteristic}

\author{Kohsuke Shibata}

\address{Graduate School of Mathematical Sciences, University of Tokyo, 3-8-1 
Komaba, Meguro-ku, 
Tokyo, 153-8914, Japan.}

\email{shibata@ms.u-tokyo.ac.jp}

\thanks{
The author is partially supported by JSPS Grant-in-Aid for Early-Career Scientists 19K14496 and the
Iwanami Fujukai Foundation.
}

\begin{abstract}
We show the existence of   prime divisors computing minimal log discrepancies in positive characteristic except for a special case.
Moreover we prove the lower semicontinuity of minimal log discrepancies  for smooth varieties in positive characteristic.

\end{abstract}

\maketitle

\section{Introduction}
The minimal log discrepancy is an important invariant of singularities in birational geometry.
The study of minimal log discrepancies in characteristic $0$ has been developed based on the resolution of singularities.
We   showed that the existence of a prime divisor computing the minimal log discrepancy using  resolution of singularities  and proved many properties of minimal log discrepancies using divisors computing minimal log discrepancies.

In positive characteristic, the existence of a prime divisor computing the minimal log discrepancy is not known in general.
The main difficulty in dealing with minimal log discrepancies in positive characteristic
is the lack of resolution of singularities.
In this paper we prove the existence of a prime divisor computing the minimal log discrepancy in arbitrary characteristic.

\begin{thm}[Theorem \ref{Main theorem}, Theorem \ref{Main theorem2}]
Let $X$ be a log canonical variety over an algebraically
closed field $k$ of  arbitrary characteristic, $W$ be a closed subset of $X$,  $\mathfrak a\subset \mathcal O_X$ be a non-zero ideal sheaf and $c\in\mathbb R_{\ge 0}$.
\begin{enumerate}
\item
If $c\neq \mathrm{lct}_W(\mathfrak a)$,
then there exists a prime divisor $E$ over $X$  computing $\mathrm{mld}(W;X,\mathfrak a^c)$.

\item
If $c<\mathrm{lct}_W(\mathfrak a)$,
then there exist  prime divisors $E_1,\dots,E_n$ over $X$  such that   $c_X(E_i)\subset W$ and
for any $s\in[0,c]$,
$$\mathrm{mld}(W;X,\mathfrak a^{s})=\min_{1\le i\le n}a(E_i;X,\mathfrak a^s).$$ 

\end{enumerate}
\end{thm}

In \cite{A}, Ambro posed the lower semicontinuity (LSC) conjecture for minimal log discrepancies.
\begin{conj}[LSC conjecture]
Let $X$ be a normal $\mathbb Q$-Gorenstein variety, $\mathfrak a\subset \mathcal O_X$ be a non-zero ideal sheaf and $c\in\mathbb R_{\ge 0}$. 
Then the map $|X| \to \mathbb R_{\ge 0}\cup\{-\infty\},\ x\mapsto \mathrm{mld}(x;X,\mathfrak a^c)$ is  lower semicontinuous, where $|X|$ is the set of all closed points of $X$.
\end{conj}

The LSC  conjecture is  not known to be true in general even in characteristic 0 and has been proved in the following  cases:
varieties over $\mathbb C$ of dimension at most $3$ or toric varieties by Ambro \cite{A};
smooth varieties  over $\mathbb C$ by Ein,  Musta\c{t}\v{a} and Yasuda  \cite{EMY};
locally complete intersection varieties over $\mathbb C$   by Ein and Musta\c{t}\v{a}  \cite{EM04};
 quotient singularities over $\mathbb C$    by Nakamura  \cite{N}.

In this paper, we will show LSC conjecture for smooth varieties in arbitrary characteristic.
\begin{thm}
Let $X$ be a smooth variety over an algebraically
closed field $k$ of  arbitrary characteristic, $\mathfrak a\subset \mathcal O_X$ be a non-zero ideal sheaf and $c\in\mathbb R_{\ge 0}$.
Suppose that $c<\mathrm{lct}(\mathfrak a)$.
Then the map $|X| \to \mathbb R_{\ge 0}$, $x\mapsto \mathrm{mld}(x;X,\mathfrak a^c)$ is  lower semicontinuous, where $|X|$ is the set of all closed points of $X$.

\end{thm}

The structure of this paper is as follows: 
In Section 2 we give the definitions  of minimal log discrepancies and jet schemes.
In Section 3 we prove the existence of a prime divisor computing the minimal log discrepancy in arbitrary characteristic.
In Section 4 we prove the LSC conjecture for smooth varieties in arbitrary characteristic using jet schemes.

\begin{ack}
The author would like to thank Professor  Shihoko Ishii  and Professor  Shunsuke Takagi for valuable conversations.
\end{ack}

\noindent 
{\bf Conventions.}
Throughout this paper,   a variety is a reduced irreducible separated scheme of
finite type over an algebraically closed field $k$ of arbitrary characteristic.

\section{Preliminaries}

In this section, we give necessary definitions for later use.

\subsection{Minimal log discrepancies}
\begin{defn}
Let $X$ be a variety.
We say that $E$ is a prime divisor over $X$, if there is a birational morphism $f:Y\to X$ such that $Y$ is normal and $E$ is a prime divisor on $Y$.
The closure of $f(E)\subset X$ is called the center of $E$ on $X$  and denoted by $c_X(E)$.
\end{defn}

\begin{defn}
Let $X$ be a $\mathbb Q$-Gorenstein normal variety , $\mathfrak a\subset \mathcal O_X$ be a non-zero ideal sheaf, $c\in\mathbb R_{\ge 0}$  and $E$ be a prime divisor over $X$.
The log discrepancy of $(X,\mathfrak a^c)$ at $E$ is defined as 
$$a(E;X,\mathfrak a^c):=k_E+1-c\mathrm{ord}_E(\mathfrak a),$$
where $k_E$ is the coefficient of the relative canonical divisor $K_{Y/X}$ at $E$.
Here $f: Y\to X$ is birational morphism with normal $Y$ such that $E$ appears on $Y$.
\end{defn}

\begin{defn}
Let $X$ be a $\mathbb Q$-Gorenstein normal variety,  $\mathfrak a\subset \mathcal O_X$ be a non-zero ideal sheaf and $c\in\mathbb R_{\ge 0}$.
For a closed subset $W$ of $X$ and for a (not necessarily closed) point $\eta\in X$, we define the  minimal log discrepancy of the pair $(X,\mathfrak a^c)$ at  $W$ and  the  minimal log discrepancy of the pair $(X,\mathfrak a^c)$ at  $\eta$ as follows:
$$\mathrm{mld}(W;X,\mathfrak a^c)=\mathrm{inf}\{a(E;X,\mathfrak a^c)\ |\ E: \mbox{prime\ divisors\ over\ } X\ \mbox{with}\ c_X(E)\subset W\},$$
$$\mathrm{mld}(\eta;X,\mathfrak a^c)=\mathrm{inf}\{a(E;X,\mathfrak a^c)\ |\ E: \mbox{prime\ divisors\ over\ } X\  \mbox{with}\ c_X(E)=\overline{\{\eta\}}\}$$
when $\dim X\geq2$.
When $\dim X=1$ and the right-hand side is $\ge 0$, then we define $\mathrm{mld}(W;X,\mathfrak a^c)$ and $\mathrm{mld}(\eta;X,\mathfrak a^c)$ by the
right-hand side.
Otherwise, we define $\mathrm{mld}(W;X,\mathfrak a^c)=-\infty$ and  $\mathrm{mld}(\eta;X,\mathfrak a^c)=-\infty$.
\end{defn}

\begin{defn}
Let  $X$ be a normal $\mathbb Q$-Gorenstein variety.
$X$ is said to be log canonical  if 
 $a(E;X,\mathcal O_X)\ge 0$ for every  prime divisor $E$ over $X$.
\end{defn}

\begin{defn}
Let $X$ be a log canonical variety,  $\mathfrak a\subset \mathcal O_X$ be a non-zero ideal sheaf and $c\in\mathbb R_{\ge 0}$.
For  a closed subset $W$ of $X$ and for a  point $\eta\in X$, we define  the  log canonical threshold of the pair $(X,\mathfrak a)$ at  $W$ and  the  log canonical threshold of the pair $(X,\mathfrak a)$ at  $\eta$  as follows:
$$\mathrm{lct}_W(X,\mathfrak a)=\mathrm{inf}\Big\{\frac{k_E+1}{\mathrm{ord}_E(\mathfrak a)}\ |\ E: \mbox{prime\ divisors\ over\ } X\ \mbox{with}\ c_X(E)\subset W\Big\},$$
$$\mathrm{lct}_\eta(X,\mathfrak a)=\mathrm{inf}\Big\{\frac{k_E+1}{\mathrm{ord}_E(\mathfrak a)}\ |\ E: \mbox{prime\ divisors\ over\ } X\ \mbox{with}\ c_X(E)=\overline{\{\eta\}}\Big\}.$$
We simply write  
$\mathrm{lct}_W(\mathfrak a)$ (resp. $\mathrm{lct}_\eta(\mathfrak a)$) instead of 
$\mathrm{lct}_W(X,\mathfrak a)$  (resp. $\mathrm{lct}_\eta(X,\mathfrak a)$)  if no confusion is possible.
If $W=X$, we write $\mathrm{lct}(\mathfrak a)$ instead of $\mathrm{lct}_X(\mathfrak a)$.
\end{defn}

\begin{defn}
Let $X$ be a $\mathbb Q$-Gorenstein normal variety, $W$ be a closed subset of $X$, $\eta$ be a point of $X$,  $\mathfrak a\subset \mathcal O_X$ be a non-zero ideal sheaf and $c\in\mathbb R_{\ge 0}$.
We say that a prime divisor $E$ over $X$ computes $\mathrm{mld}(W;X,\mathfrak a^c)$ (resp. $\mathrm{mld}(\eta;X,\mathfrak a^c)$) if the center of $E$ is contained in $W$ (resp. is $\overline{\{\eta\}}$) and  either 
$$a(E;X,\mathfrak a^c)=\mathrm{mld}(W;X,\mathfrak a^c)\ \ (\mbox{resp.}\  a(E;X,\mathfrak a^c)=\mathrm{mld}(\eta;X,\mathfrak a^c))  \ \  \mathrm{or}$$
$$a(E;X,\mathfrak a^c)<0. $$
\end{defn}

\subsection{Jet schemes and arc spaces}
We briefly review in this subsection  jet schemes and arc spaces.
The reader is referred to \cite{EM} and   \cite{IR2} for details.

Let  $X$ be a scheme of  finite type over $k$, ${\mathcal S}ch/k $ be the category of $ k $-schemes  
   and $ {\mathcal S}et $ be the category of sets.
  Define a contravariant functor  $ F_{m}: {\mathcal S}ch/k \to {\mathcal S}et $
  by 
$$
 F_{m}(Y)=\mathrm{Hom} _{k}(Y\times_{\mathrm{Spec} k}\mathrm{Spec} k[t]/(t^{m+1}), X).
$$
Then, $ F_{m} $ is representable by a scheme $ X_m $ of finite
  type over $ k $.
  The scheme $ X_m $ is  called the {\it $m$-jet scheme}  of $ X $.
The canonical surjection $ k[t]/(t^{m+1})\to k[t]/(t^{n+1}) $ $(n<m< \infty)$
  induces a morphism $ \psi_{mn}:X_m\to X_n $.
 There exists the projective limit $$X_\infty:=\lim_{\overleftarrow {m}} X_m$$
   and it is called the {\it arc space} of $X$.
   There is  a bijection for a $k$-algebra A  as follows:
   $$
 \mathrm{Hom} _{k}(\mathrm{Spec} A, X_{\infty})\simeq\mathrm{Hom} _{k}(\mathrm{Spec} A[[t]], X).
$$

\begin{defn}
 For a scheme $X$ of  finite type over $k$, let $X_m$ $(m\in \mathbb Z_{\ge 0})$ and $X_\infty$ be
 the $m$-jet scheme and the arc space of $X$.
 Denote the canonical truncation morphisms by
 $\psi_m: X_\infty\to X_m$ and $\pi_m: X_m\to X$.
 To specify the space $X$, we sometimes write $\psi^X_m$ and $\pi^X_m$.
\end{defn}

\begin{defn}
For an arc $\gamma\in X_\infty$, 
the order of an ideal $\mathfrak a\subset \mathcal O_X$ measured by $\gamma$
is defined as follows:
let $\gamma^*: \mathcal O_{X, \gamma((t))} \to k[[t]]$
be the corresponding
ring homomorphism of $\gamma$.
Then, we define
$$\mathrm{ord}_\gamma(\mathfrak a)=\sup \{ r\in \mathbb Z_{\geq 0}\mid \gamma^*(\mathfrak a)\subset (t^r)\}.$$
We define the subsets ``contact loci" in the arc space as follows:
 $$\mathrm{Cont}^{\geq m}(\mathfrak a)=\{\gamma \in X_\infty \mid \mathrm{ord}_\gamma(\mathfrak a)\geq m\}.$$
By this definition, we can see that
$$\mathrm{Cont}^{\geq m}(\mathfrak a)=\psi_{m-1}^{-1}(Z(\mathfrak a)_{m-1}),$$
where $Z(\mathfrak a)$ is the closed subscheme defined by the ideal $\mathfrak a$ in $X$.
We can define in the same way the subset $\mathrm{Cont}^{\geq m}(\mathfrak a)_n$ (if $m\le n+1$) of $X_n$ and we have 
$$\mathrm{Cont}^{\geq m}(\mathfrak a)=\psi_{n}^{-1}(\mathrm{Cont}^{\ge m}(\mathfrak a)_{n}).$$
Note that $\mathrm{Cont}^{\ge m+1}(\mathfrak a)_{m}=Z(\mathfrak a)_m$.
\end{defn}

\begin{defn}
\label{defeta}
Let $X$ be a smooth variety,  $\eta\in X$ be a (not necessarily closed) point and  $\mathfrak a\subset \mathcal O_X$ be a non-zero  ideal sheaf.
We define the codimension of $\mathrm{Cont}^{\geq m}(\mathfrak a)\cap \pi_\infty^{-1}(\eta)$ and 
 the codimension of $\mathrm{Cont}^{\geq m}(\mathfrak a)_{n}\cap \pi_{n}^{-1}(\eta)$ for $m\le n+1$
as follows:
$$\mathrm{codim} (\mathrm{Cont}^{\geq m}(\mathfrak a)\cap \pi_\infty^{-1}(\eta))
$$
$$
:=\inf\left\{\mathrm{codim} (C, X_\infty)\mid \begin{array}{l}C \mbox{\ is \ an\  irreducible \ component \ of}\\
 {\mathrm{Cont}^{\geq m}(\mathfrak a)\cap\pi_\infty^{-1}(\overline{\{\eta\}})}\ \ 
\mbox{dominating}\ \overline{\{\eta\}}\end{array}\right\},
$$
$$\mathrm{codim} (\mathrm{Cont}^{\geq m}(\mathfrak a)_n\cap \pi_{n}^{-1}(\eta))
$$
$$
:=\inf\left\{\mathrm{codim} (C, X_n)\mid \begin{array}{l}C \mbox{\ is \ an\  irreducible \ component \ of}\\
 {\mathrm{Cont}^{\geq m}(\mathfrak a)_n\cap\pi_n^{-1}(\overline{\{\eta\}})}\ \ 
\mbox{dominating}\ \overline{\{\eta\}}\end{array}\right\}.
$$
\end{defn}

\begin{thm}[{\cite[Theorem 3.18]{IR2}}]\label{mld contact in IR}
Let  $X$ be a smooth variety $X$,   
 $\mathfrak a\subset \mathcal O_X$ be a non-zero ideal sheaf
and $c\in\mathbb R_{\ge 0}$.
 Let $\eta\in X$ be  a (not necessarily closed) point, $W$ be a proper closed subset of $X$ and let $I_W$ be the defining ideal of $W$.
Then,
$$\mathrm{mld}(W;X,\mathfrak a^c)=\inf_{m\in\mathbb Z_{\ge 0}}\{\mathrm{codim}(\mathrm{Cont}^{\geq m}(\mathfrak a)\cap\mathrm{Cont}^{\geq 1}(I_W))-cm\},$$
$$\mathrm{mld}(\eta;X,\mathfrak a^c)=\inf_{m\in\mathbb Z_{\ge 0}}\{\mathrm{codim}(\mathrm{Cont}^{\geq m}(\mathfrak a)\cap\pi_\infty^{-1}(\eta))-cm\}.$$

\end{thm}

\begin{rem}
In \cite{IR2}, the above theorem is stated under more general setting. 
In this paper, we need only this form.
Since $X$ is smooth, by the definition of the Mather-Jacobian log discrepancy, we have $\mathrm{mld}(W,X,\mathfrak a^c)=\mathrm{mld}_\mathrm{MJ}(W;X,\mathfrak a^c)$, where $\mathrm{mld}_\mathrm{MJ}(W;X,\mathfrak a^c)$ is the Mather-Jacobian minimal log discrepancy, on which we do not give here the definition because we do not use it. 
\end{rem}

\section{Divisors computing minimal log discrepancies}
In this section,  we prove the existence of a prime divisor computing the minimal log discrepancy in arbitrary characteristic.

\begin{lem}\label{mld=-infty}
Let $X$ be a $\mathbb Q$-Gorenstein normal  variety, $W$ be a closed subset of $X$,  $\eta\in X$ be a point, $\mathfrak a\subset \mathcal O_X$ be a non-zero ideal sheaf and $c\in\mathbb R_{\ge 0}$.
\begin{enumerate}
\item
If $\mathrm{mld}(W;X,\mathfrak a^c)=-\infty$, then there exists a prime divisor $E$ over $X$  computing $\mathrm{mld}(W;X,\mathfrak a^c)$. 
\item
If $\mathrm{mld}(\eta;X,\mathfrak a^c)=-\infty$, then there exists a prime divisor $E$ over $X$  computing $\mathrm{mld}(\eta;X,\mathfrak a^c)$. 
\end{enumerate}
\end{lem}

\begin{proof}$(1)$
Since  $\mathrm{mld}(W;X,\mathfrak a^c)=-\infty$, there exists a prime divisor $E$ over $X$ such that 
$a(E; X,\mathfrak a^c)<0$ and $c_X(E)\subset W$.
Therefore $E$ computes $\mathrm{mld}(W;X,\mathfrak a^c)$.

The proof of $(2)$ follows in the same way.
\end{proof}

\begin{lem}\label{mld rational}
Let $X$ be a $\mathbb Q$-Gorenstein normal  variety, $W$ be a closed subset of $X$, $\eta\in X$ be a point,  $\mathfrak a\subset \mathcal O_X$ be a non-zero ideal sheaf and $c\in \mathbb Q_{\ge 0}$.
 Then \begin{enumerate}
\item
There exists a prime divisor $E$ over $X$  computing $\mathrm{mld}(W;X,\mathfrak a^c)$.
\item
There exists a prime divisor $E$ over $X$  computing $\mathrm{mld}(\eta;X,\mathfrak a^c)$.
\end{enumerate}

\end{lem}

\begin{proof}$(1)$
We may assume that $\mathrm{mld}(W;X,\mathfrak a^c)\ge 0$ by Lemma \ref{mld=-infty}.
Let $r=\mathrm{min}\{r\in\mathbb N \mid rK_X\ \mbox{is\ Cartier} \}$ and $p,q\in\mathbb Z_{\ge 0}$ with $c=\frac{p}{q}$.
Since $a(E;X,\mathfrak a^c)\in \frac{1}{rq}\mathbb Z_{\ge 0}$ for any divisor $E$ over $X$  with $c_X(E)\subset W$,
the set $\{a(E;X,\mathfrak a^c) \mid  c_X(E)\subset W\}$ satisfies the descending chain condition.
Therefore there exists a prime divisor $E$ over $X$  computing $\mathrm{mld}(W;X,\mathfrak a^c)$.

The proof of $(2)$ follows in the same way.
\end{proof}

\begin{prop} \label{lct mld}
Let $X$ be a log canonical variety, $W$ be a closed subset of $X$, $\eta\in X$ be a point,  $\mathfrak a\subset \mathcal O_X$ be a non-zero ideal sheaf and $c\in\mathbb R_{\ge 0}$.
\begin{enumerate}
\item
If $c\le\mathrm{lct}_W(\mathfrak a)$, then $\mathrm{mld}(W;X,\mathfrak a^c)\ge0$.
\item
If $c>\mathrm{lct}_W(\mathfrak a)$, then $\mathrm{mld}(W;X,\mathfrak a^c)=-\infty$.
\item
If $c\le\mathrm{lct}_\eta(\mathfrak a)$, then $\mathrm{mld}(\eta;X,\mathfrak a^c)\ge0$.
\item
If $c>\mathrm{lct}_\eta(\mathfrak a)$, then $\mathrm{mld}(\eta;X,\mathfrak a^c)=-\infty$.
\end{enumerate}
\end{prop}

\begin{proof}(1)
For any prime divisor $E$ over $X$ with $c_X(E)\subset W$, 
$$c\le\mathrm{lct}_W(\mathfrak a)\le \frac{k_E+1}{\mathrm{ord}_E(\mathfrak a)}.$$
This implies that $a(E;X,\mathfrak a^c)\ge 0$ for any prime divisor $E$ over $X$ with $c_X(E)\subset W$.
Therefore $\mathrm{mld}(W;X,\mathfrak a^c)\ge0$.

(2) Since $c>\mathrm{lct}_W(\mathfrak a)$, there exists a prime divisor $E$ over $X$ such that  $c_X(E)\subset W$ and
$c>\frac{k_E+1}{\mathrm{ord}_E(\mathfrak a)}.$
This implies that $a(E;X,\mathfrak a^c)< 0$.
Thus $\mathrm{mld}(W;X,\mathfrak a^c)=-\infty$.

The proofs of $(3),(4)$ follow in the same way.
\end{proof}

\begin{thm}\label{Main theorem}
Let $X$ be a log canonical variety, $W$ be a closed subset of $X$, $\eta\in X$ be a point,   $\mathfrak a\subset \mathcal O_X$ be a non-zero ideal sheaf and $c\in\mathbb R_{\ge 0}$.
\begin{enumerate}
\item If  $c\neq\mathrm{lct}_W(\mathfrak a)$, then
there exists  a prime  divisor $E$ over $X$  computing $\mathrm{mld}(W;X,\mathfrak a^c)$.
\item If  $c\neq\mathrm{lct}_{\eta}(\mathfrak a)$, then
there exists a prime  divisor $E$ over $X$  computing $\mathrm{mld}(\eta;X,\mathfrak a^c)$.
\end{enumerate}

\end{thm}

\begin{proof}$(1)$
If $c>\mathrm{lct}_W(\mathfrak a)$, then this follows from  Lemma \ref{mld=-infty} and Proposition \ref{lct mld}.

We assume that $c<\mathrm{lct}_W(\mathfrak a)$.
If the conclusion of this theorem fails, then 
$a(E;X,\mathfrak a^c)>\mathrm{mld}(W;X,\mathfrak a^c)$
for any prime divisor $E$ over $X$ with $c_X(E)\subset W$.
Note that the set $\{k_E \mid  c_X(E)\subset W\}$ and the set $\{\mathrm{ord}_E(\mathfrak a) \mid  c_X(E)\subset W\}$ satisfy the descending chain condition.
We can find the set of prime divisors $\{E_i\}_{i\in\mathbb N}$ over $X$ such that for  $i\in\mathbb N$,
$$c_X(E_i)\subset W,\ \ a(E_i;X,\mathfrak a^c)>a(E_{i+1};X,\mathfrak a^c),\ \ \mathrm{ord}_{E_{i}}(\mathfrak a)<\mathrm{ord}_{E_{i+1}}(\mathfrak a)\ \   \mbox{and}$$
$$\mathrm{mld}(W;X,\mathfrak a^c)=\lim_{i\to\infty}a(E_i;X,\mathfrak a^c).$$
Since $c<\mathrm{lct}_W(\mathfrak a)$, there exists $\delta>0$ such that $c+\delta<\mathrm{lct}_W(\mathfrak a)$ and $c+\delta\in\mathbb Q_{>0}.$
By Lemma \ref{mld rational}, there exists a prime divisor $F$ over $X$ computing $\mathrm{mld}(W;X,\mathfrak a^{c+\delta})$.
Then we have for $i\in\mathbb N$,
$$a(E_i;X,\mathfrak a^{c+\delta})\ge a(F;X,\mathfrak a^{c+\delta})=\mathrm{mld}(W;X,\mathfrak a^{c+\delta}).$$
This implies that 
$$a(E_i;X,\mathfrak a^c)-\delta(\mathrm{ord}_{E_i}(\mathfrak a)-\mathrm{ord}_F(\mathfrak a))\ge a(F;X,\mathfrak a^c).$$
Since $\mathrm{ord}_{E_{i}}(\mathfrak a)<\mathrm{ord}_{E_{i+1}}(\mathfrak a)$ and $\mathrm{ord}_{E_i}(\mathfrak a)\in\mathbb Z_{\ge 0}$, there exists $j\in\mathbb N$ such that 
 $\mathrm{ord}_{E_i}(\mathfrak a)\ge \mathrm{ord}_{F}(\mathfrak a)$ for any $i$ with $i\ge j$.
Therefore for any  $i$ with $i\ge j$,
$$a(E_i;X,\mathfrak a^c)\ge a(F;X,\mathfrak a^c)\ge\mathrm{mld}(W;X,\mathfrak a^c).$$
By the squeeze theorem, we have 
$a(F;X,\mathfrak a^c)=\mathrm{mld}(W;X,\mathfrak a^c),$
which is a contradiction.
Hence there exists a prime divisor $E$ over $X$ computing $\mathrm{mld}(W;X,\mathfrak a^c)$.

The proof of $(2)$ follows in the same way.
\end{proof}

\begin{rem}
This theorem can be easily extended for the combination of ideals $\mathfrak a_1,\dots,\mathfrak a_n$ instead of one ideal $\mathfrak a$. I.e., we have if  $1\neq \mathrm{lct}_W(\mathfrak a_1^{c_1}\cdots\mathfrak a_n^{c_n})$,
 then there exists a prime divisor $E$ over $X$  computing $\mathrm{mld}(W;X,\mathfrak a_1^{c_1}\cdots\mathfrak a_n^{c_n})$.
\end{rem}

\begin{cor}
Let $X$ be a log canonical variety, $W$ be a closed subset of $X$, $I_W$ be the defining ideal of $W$,  $\mathfrak a\subset \mathcal O_X$ be a non-zero ideal sheaf and $c\in\mathbb R_{\ge 0}$.
Suppose that $c\neq\mathrm{lct}_W(\mathfrak a)$.
 Then there exists a natural number $m$  such that 
$$\mathrm{mld}(W;X,\mathfrak a^c)=\mathrm{mld}(W;X,(\mathfrak a+ I_W^m)^c).$$
\end{cor}

\begin{proof}
Since $\mathfrak a\subset\mathfrak a+ I_W^n$ for any $n\in\mathbb N$, we have
$\mathrm{mld}(W;X,\mathfrak a^c)\le\mathrm{mld}(W;X,(\mathfrak a+ I_W^n)^c)$.
By Theorem \ref{Main theorem}, there exists a prime divisor $E$ over $X$  computing $\mathrm{mld}(W;X,\mathfrak a^c)$.
For $m\in\mathbb N$ with $m>\mathrm{ord}_E(\mathfrak a)$, we have 
$\mathrm{ord}_E(\mathfrak a)=\mathrm{ord}_E(\mathfrak a+I_W^m)$.
Therefore for $m\in\mathbb N$ with $m>\mathrm{ord}_E(\mathfrak a)$, we have 
$$a(E;X,\mathfrak a^c)=a(E;X,(\mathfrak a+ I_W^m)^c)\ge\mathrm{mld}(W;X,(\mathfrak a+ I_W^m)^c).$$
Hence this corollary holds.
\end{proof}

\begin{lem}\label{compute mld and mld+}
Let $X$ be a log canonical variety, $W$ be a closed subset of $X$,  $\mathfrak a\subset \mathcal O_X$ be a non-zero ideal sheaf and $c\in\mathbb R_{\ge 0}$.
Suppose that $c<\mathrm{lct}_W(\mathfrak a)$.
Then,
\begin{enumerate}
\item  There exist $\delta>0$ and a prime divisor  $E$ over $X$ such that $c+\delta<\mathrm{lct}_W(\mathfrak a)$ and 
$E$ computes $\mathrm{mld}(W;X,\mathfrak a^c)$ and $\mathrm{mld}(W;X,\mathfrak a^{c+\delta})$. 

\item Assume that $c>0$.
There exist $\delta>0$ and a prime divisor  $E$ over $X$ such that $0<c-\delta$ and 
$E$ computes $\mathrm{mld}(W;X,\mathfrak a^c)$ and $\mathrm{mld}(W;X,\mathfrak a^{c-\delta})$. 

\end{enumerate}
\end{lem}

\begin{proof}
$(1)$ Let $\{s_i\}_{i\in\mathbb N}$ be a sequence of positive real numbers  such that 
$$c+s_i<\mathrm{lct}_W(\mathfrak a),\ \ s_i>s_{i+1}\ \ \mbox{for}\ \ i\in\mathbb N \ \ \mbox{and}\ \ \lim_{i\to\infty}s_i=0.$$
Let $E$ and $E_i$ be  prime divisors over $X$ computing $\mathrm{mld}(W;X,\mathfrak a^c)$ and $\mathrm{mld}(W;X,\mathfrak a^{c+s_i})$, respectively. Then for $i\in\mathbb N$,
$$a(E;X,\mathfrak a^{c+s_i})\ge a(E_{i};X,\mathfrak a^{c+s_i})\ \ \mbox{and}\ \ a(E_{i+1};X,\mathfrak a^{c+s_i})\ge a(E_{i};X,\mathfrak a^{c+s_i}).$$
Hence we have
$$a(E;X,\mathfrak a^c)-s_i(\mathrm{ord}_{E}(\mathfrak a)-\mathrm{ord}_{E_i}(\mathfrak a))\ge a(E_i;X,\mathfrak a^c)\ \ \mbox{and}$$
$$a(E_{i+1};X,\mathfrak a^{c+s_{i+1}})-(s_i-s_{i+1})(\mathrm{ord}_{E_{i+1}}(\mathfrak a)-\mathrm{ord}_{E_i}(\mathfrak a))\ge a(E_i;X,\mathfrak a^{c+s_{i+1}}).$$
Since $E$ and $E_{i+1}$ compute $\mathrm{mld}(W;X,\mathfrak a^c)$ and $\mathrm{mld}(W;X,\mathfrak a^{c+s_{i+1}})$, respectively, 
we have for $i\in\mathbb N$,
$$\mathrm{ord}_E(\mathfrak a)\le\mathrm{ord}_{E_{i+1}}(\mathfrak a)\le\mathrm{ord}_{E_i}(\mathfrak a).$$
Note that  $\mathrm{mld}(W;X,\mathfrak a^{c+s_i})\le\mathrm{mld}(W;X,\mathfrak a^{c+s_{i+1}})\le \mathrm{mld}(W;X,\mathfrak a^c)$.
Let $r=\mathrm{min}\{r\in\mathbb N \mid rK_X\ \mbox{is\ Cartier} \}$.
We can find $k\in\frac{1}{r}\mathbb Z$, $m\in\mathbb Z_{\ge0}$ and a sequence of natural numbers $\{a_j\}_{j\in\mathbb N}$    such that for $j\in\mathbb N$,
 $$a_j<a_{j+1},\ \ m=\mathrm{ord}_{E_{a_j}}(\mathfrak a)\ \ \mbox{and}\ \  k=k_{E_{a_j}}.$$
Note that $a(E_{a_j};X,\mathfrak a^c)=k+1-cm$.
Since $$a(E;X,\mathfrak a^c)\le a(E_{a_j};X,\mathfrak a^c)\le a(E;X,\mathfrak a^c)-s_{a_j}(\mathrm{ord}_{E}(\mathfrak a)-\mathrm{ord}_{E_{a_j}}(\mathfrak a)),$$
by the squeeze theorem, we have 
$a(E;X,\mathfrak a^c)=a(E_{a_j};X,\mathfrak a^c)$.
Therefore $E_{a_j}$ computes  $\mathrm{mld}(W;X,\mathfrak a^c)$ and $\mathrm{mld}(W;X,\mathfrak a^{c+s_{a_j}})$.

The proof of $(2)$ follows in the same way.
\end{proof}

\begin{prop}\label{between mld and mld+}
Let $X$ be a log canonical variety, $W$ be a closed subset of $X$,  $\mathfrak a\subset \mathcal O_X$ be a non-zero ideal sheaf and $c\in\mathbb R_{\ge 0}$.
Suppose that $c<\mathrm{lct}_W(\mathfrak a)$.
Then,
\begin{enumerate}
\item There exist $\delta>0$ and a prime divisor $E$ over $X$ such that $c+\delta<\mathrm{lct}_W(\mathfrak a)$ and 
for any $s\in[0,\delta]$,
$E$ computes  $\mathrm{mld}(W;X,\mathfrak a^{c+s})$. 

\item Assume that $c>0$. Then there exist $\delta>0$ and a prime divisor $E$ over $X$ such that $0<c-\delta$ and 
for any $s\in[0,\delta]$,
$E$ computes  $\mathrm{mld}(W;X,\mathfrak a^{c-s})$. 
\end{enumerate}
\end{prop}

\begin{proof}$(1)$
By Lemma \ref{compute mld and mld+},
 there exist $\delta>0$ and a prime divisor  $E$ over $X$ such that $c+\delta<\mathrm{lct}_W(\mathfrak a)$ and 
$E$ computes $\mathrm{mld}(W;X,\mathfrak a^c)$ and $\mathrm{mld}(W;X,\mathfrak a^{c+\delta})$. 
Let $s$ be a positive real number with $s\in (0,\delta)$.
By Theorem \ref{Main theorem}, there exists a prime divisor $E_s$  over $X$ computing $\mathrm{mld}(W;X,\mathfrak a^{c+s})$.
Then 
$$a(E;X,\mathfrak a^{c+s})\ge a(E_{s};X,\mathfrak a^{c+s})\ \ \mbox{and}\ \ a(E_s;X,\mathfrak a^{c+\delta})\ge a(E;X,\mathfrak a^{c+\delta}).$$
These inequalities imply that 
\begin{align*}
&a(E_s;X,\mathfrak a^c)-a(E;X,\mathfrak a^c)\ge \delta(\mathrm{ord}_{E_s}(\mathfrak a)-\mathrm{ord}_{E}(\mathfrak a))\\
&\ge s(\mathrm{ord}_{E_s}(\mathfrak a)-\mathrm{ord}_{E}(\mathfrak a))\ge a(E_s;X,\mathfrak a^c)-a(E;X,\mathfrak a^c).
\end{align*}
Therefore we have $\mathrm{ord}_{E_s}(\mathfrak a)=\mathrm{ord}_{E}(\mathfrak a)$ and $a(E_s;X,\mathfrak a^c)=a(E;X,\mathfrak a^c).$
Hence $a(E_s;X,\mathfrak a^{c+s})=a(E;X,\mathfrak a^{c+s})$.
This implies that for any $s\in[0,\delta]$,
$E$ computes  $\mathrm{mld}(W;X,\mathfrak a^{c+s})$. 

The proof of $(2)$ follows in the same way.
\end{proof}

\begin{thm}\label{Main theorem2}
Let $X$ be a log canonical variety, $W$ be a closed subset of $X$,  $\mathfrak a\subset \mathcal O_X$ be a non-zero ideal sheaf and $c\in\mathbb R_{\ge 0}$.
Suppose that $c<\mathrm{lct}_W(\mathfrak a)$.
Then there exist  prime divisors $E_1,\dots,E_n$ over $X$  such that   $c_X(E_i)\subset W$ and
for any $s\in[0,c]$,
$$\mathrm{mld}(W;X,\mathfrak a^{s})=\min_{1\le i\le n}a(E_i;X,\mathfrak a^s).$$ 
\end{thm}

\begin{proof}
By Proposition \ref{between mld and mld+},
for any $t\in [0,c]$,
there exist $\delta_t>0$ and prime divisors $F_1$ and $F_2$ over $X$ such that  $c_X(F_i)\subset W$ and
for any $s\in [t-\delta_t,t+\delta_t]\cap[0,c]$,
$$\mathrm{mld}(W;X,\mathfrak a^s)=\mathrm{min}\{a(F_1;X,\mathfrak a^s),\ a(F_2;X,\mathfrak a^s)\}.$$
Since $[0,c]$ is a compact set, there exist prime  divisors $E_1,\dots,E_n$ over $X$ such that   $c_X(E_i)\subset W$ and
for any $s\in[0,c]$,
$$\mathrm{mld}(W;X,\mathfrak a^{s})=\min_{1\le i\le n}a(E_i;X,\mathfrak a^s).$$ 
\end{proof}


\section{Semicontinuity of minimal log discrepancies}
In this section,  we prove the LSC conjecture for smooth varieties  using jet schemes.
We remark that  in this section we always assume that  $c<\mathrm{lct}(\mathfrak a)$  in order to be able to apply Theorem  \ref{Main theorem}.

\begin{lem}
\label{mld=s_l}  
Let $X$ be a smooth variety,  $\eta\in X$ be a (not necessarily closed)   point, $\mathfrak a\subset \mathcal O_X$ be a non-zero ideal sheaf and $c\in\mathbb R_{\ge 0}$.
Suppose that $c<\mathrm{lct}(\mathfrak a)$.
Let $E$ be a prime divisor over $X$ computing  $\mathrm{mld}(\eta; X,\mathfrak a^c)$ and $l=\mathrm{ord}_E(\mathfrak a)$.
Then, we have 
    $$\mathrm{mld}(\eta;X,\mathfrak a^{c})=\mathrm{codim}(\mathrm{Cont}^{\ge l}(\mathfrak a)\cap\pi_\infty^{-1}(\eta))-cl.$$
\end{lem}

\begin{proof}
By Theorem \ref{mld contact in IR},  we have
$$\mathrm{mld}(\eta;X,\mathfrak a^{c})\le\mathrm{codim}(\mathrm{Cont}^{\ge l}(\mathfrak a)\cap\pi_\infty^{-1}(\eta))-cl.$$

Let  $f: Y\to X$  be a birational morphism such that  $Y$ is normal and  $E$ appears on $Y$.
Let $p$ be the generic point of $E$, $f_\infty:Y_\infty\to X_\infty$ be the morphism of arc spaces corresponding
to $f$ and $C_X(\mathrm{ord}_E)=\overline{f_\infty\left((\pi_\infty^Y)^{-1}(p)\right)}$.
Then  $\mathrm{codim}(C_X(\mathrm{ord}_E))=k_E+1$ by \cite[Theorem 3.13]{IR2}.
Therefore we have 
\begin{align*}
\mathrm{codim}(\mathrm{Cont}^{\ge l}(\mathfrak a)\cap\pi_\infty^{-1}(\eta))-cl&\le\mathrm{codim}(C_X(\mathrm{ord}_E))-cl
=k_E+1-c\mathrm{ord}_E(\mathfrak a)\\
&=a(E;X,\mathfrak a^c)=\mathrm{mld}(\eta;X,\mathfrak a^{c}).
\end{align*}
\end{proof}

\begin{lem}\label{global upper bound lemma}
Let $X$ be a smooth variety, $\mathfrak a\subset \mathcal O_X$ be a non-zero ideal sheaf and $c\in\mathbb R_{\ge 0}$.
Suppose that $c<\mathrm{lct}(\mathfrak a)$.
Then there exists $l\in \mathbb N$ such that 
\[l\ge\sup\Big\{\mathrm{ord}_E(\mathfrak a) \mid \eta\in X,\ E\ \mbox{computes}\ \mathrm{mld}(\eta;X,\mathfrak a^{c})\Big\}.\]
\end{lem}

\begin{proof}
Assume that
\[\sup\Big\{\mathrm{ord}_E(\mathfrak a) \mid \eta\in X,\ E\ \mbox{computes}\ \mathrm{mld}(\eta;X,\mathfrak a^{c})\Big\}=\infty.\]
Then there exist  points $\{\eta_i\}_{i\in\mathbb N}$ of $X$ and prime divisors $\{E_i\}_{i\in\mathbb N}$ over $X$ such that 
$E_i$ computes $\mathrm{mld}(\eta_i;X,\mathfrak a^{c})$ and $\lim_{i\to\infty}\mathrm{ord}_{E_i}(\mathfrak a)=\infty$.
Let $\delta$ be a positive real number with $c+\delta<\mathrm{lct}(\mathfrak a)$.
Note that $\mathrm{mld}(\eta_i;X,\mathfrak a^{c})\le\mathrm{mld}(\eta_i;X,\mathcal O_X)\le\mathrm{dim}X$ (See \cite[Corollary 3.27]{IR2}).
For $i$ with $\delta\mathrm{ord}_{E_i}(\mathfrak a)>\mathrm{dim}X$,
\begin{align*}
\mathrm{mld}(\eta_i;X,\mathfrak a^{c+\delta})&\le k_{E_i}+1-(c+\delta)\mathrm{ord}_{E_i}(\mathfrak a)=\mathrm{mld}(\eta_i;X,\mathfrak a^{c})-\delta\mathrm{ord}_{E_i}(\mathfrak a)<0,
\end{align*}
which is a contradiction to $\mathrm{mld}(\eta_i;X,\mathfrak a^{c+\delta})\ge 0$.
Therefore this lemma holds.
\end{proof}

\begin{prop}\label{mld=min}
Let $X$ be a smooth variety, $\mathfrak a\subset \mathcal O_X$ be a non-zero ideal sheaf and  $c\in\mathbb R_{\ge 0}$.
Suppose that $c<\mathrm{lct}(\mathfrak a)$.
Then there exists  $l\in \mathbb N$ such that 
for any  point $\eta\in X$,
$$\mathrm{mld}(\eta;X,\mathfrak a^{c})=\min_{0\le m\le l}\Big\{\mathrm{codim}(\mathrm{Cont}^{\ge m}(\mathfrak a)\cap\pi_\infty^{-1}(\eta))-cm\Big\}.$$
\end{prop}

\begin{proof}
This follows immediately from Theorem \ref{mld contact in IR},  Lemma \ref{mld=s_l}   and Lemma \ref{global upper bound lemma}.
\end{proof}

Recall that $k^*$ has a natural action on jet schemes.
Let  $X$ be a scheme of  finite type over $k$.
Consider $k^*=\Bbb A^1\setminus \{0\}=\mathrm{Spec} k[s,s^{-1}] $ as a multiplicative group scheme. 
For $m\in \mathbb Z_{\ge 0}$, the morphism $k[t]/(t^{m+1}) \to  k[s,s^{-1}, t]/(t^{m+1})$  defined by $t\mapsto s\cdot t$ 
gives an action  \[ \mu_{m}: k^* \times_{\mathrm{Spec }k} \mathrm{Spec }k[t]/(t^{m+1})\to \mathrm{Spec }k[t]/(t^{m+1})\] of 
  $k^* $ on $\mathrm{Spec }k[t]/(t^{m+1}) $.
 Therefore, it gives an action 
  \[ \mu_{X_m}: k^*\times_{\mathrm{Spec }k}X_{m}\to X_{m} \]
  of $k^* $ on $ X_{m} $.

For a closed point $x\in X$, we denote by $\mathfrak m_x$ the defining ideal of $x$.

\begin{lem}\label{s_m semi conti}
Let $X$ be a smooth variety, $\mathfrak a\subset \mathcal O_X$ be a non-zero ideal sheaf and $m\in\mathbb Z_{\ge 0}$.
Then the map $|X| \to \mathbb R_{\ge 0},\ x\mapsto \mathrm{codim}(\mathrm{Cont}^{\ge m}(\mathfrak a)\cap\mathrm{Cont}^{\ge 1}(\mathfrak m_x))$ is  lower semicontinuous, where $|X|$ is the set of all closed points of $X$.

\end{lem}

\begin{proof}
If $m=0$, then 
$\mathrm{codim}(\mathrm{Cont}^{\ge m}(\mathfrak a)\cap\mathrm{Cont}^{\ge 1}(\mathfrak m_x))=\mathrm{dim}X$ for any closed point $x\in X$. 
Hence  this lemma holds when $m=0$. 

We assume that $m\ge 1$.
Note that $$\mathrm{codim}(\mathrm{Cont}^{\ge m}(\mathfrak a)\cap\mathrm{Cont}^{\ge 1}(\mathfrak m_x))=\mathrm{codim}(\mathrm{Cont}^{\ge m}(\mathfrak a)_{m-1}\cap\mathrm{Cont}^{\ge 1}(\mathfrak m_x)_{m-1}).$$
Therefore it is enough to show that the map $$|X| \to \mathbb R_{\ge 0},\ x\mapsto \mathrm{dim}(\mathrm{Cont}^{\ge m}(\mathfrak a)_{m-1}\cap\mathrm{Cont}^{\ge 1}(\mathfrak m_x)_{m-1})$$ is  upper semicontinuous.
Let $\phi: \mathrm{Cont}^{\ge m}(\mathfrak a)_{m-1}\to X_{m-1}$ be the closed immersion and $\psi=\pi_{m-1}\circ\phi$.
Then for every $n\in \mathbb Z_{\ge 0}$,
$$F_n:=\{y\in \mathrm{Cont}^{\ge m}(\mathfrak a)_{m-1}\ |\ \mathrm{dim}\psi^{-1}(\psi(y))\ge n\}$$
 is  a closed subset by \cite[Theorem 14.110]{GW}.
Then $F_n$ is $k^*$-invariant.
Indeed, let $K$ be a field with $k\subset K$, $\gamma:\mathrm{Spec}K[t]/(t^m)\to X$,   $a\in k^*$ and $a^*:\mathrm{Spec}K[t]/(t^m)\to \mathrm{Spec}K[t]/(t^m)$ be the morphism induced by $t\to at$.
Then $\mu_{X_{m-1}}(a,\gamma)=\gamma\circ a^*$.
This implies that $\pi_{m-1}\mu_{X_{m-1}}(k^*\times\{\gamma\})=\pi_{m-1}(\gamma)$ and $\mathrm{ord}_\gamma(\mathfrak a)=\mathrm{ord}_{\mu_{X_{m-1}}(a,\gamma)}(\mathfrak a)$.
Therefore $F_n$ is $k^*$-invariant.

 By \cite[Proposition 3.2]{I action},
$$\psi(F_n)=\{x\in X\ |\ \mathrm{dim}{\psi^{-1}}(x)\ge n\}$$
is a closed subset.
Note that  
$\psi^{-1}(x)=\mathrm{Cont}^{\ge m}(\mathfrak a)_{m-1}\cap\mathrm{Cont}^{\ge 1}(\mathfrak m_x)_{m-1}$.
Thus the map $$|X| \to \mathbb R_{\ge 0},\ x\mapsto \mathrm{dim}(\mathrm{Cont}^{\ge m}(\mathfrak a)_{m-1}\cap\mathrm{Cont}^{\ge 1}(\mathfrak m_x)_{m-1})$$ is  upper semicontinuous.
\end{proof}

\begin{thm}\label{semicontinuity for mld}
Let $X$ be a smooth variety, $\mathfrak a\subset \mathcal O_X$ be a non-zero ideal sheaf and  $c\in\mathbb R_{\ge 0}$. 
Suppose that $c<\mathrm{lct}(\mathfrak a)$.
Then the map $|X| \to \mathbb R_{\ge 0}$, $x\mapsto \mathrm{mld}(x;X,\mathfrak a^c)$ is  lower semicontinuous,
where $|X|$ is the set of all closed points of $X$.

\end{thm}

\begin{proof}
Note that  since $\pi^{-1}_\infty(x)=\mathrm{Cont}^{\ge 1}(\mathfrak m_x)$  for any closed point $x\in X$,
$$\mathrm{codim}(\mathrm{Cont}^{\ge m}(\mathfrak a)\cap\pi^{-1}_\infty(x))=\mathrm{codim}(\mathrm{Cont}^{\ge m}(\mathfrak a)\cap\mathrm{Cont}^{\ge 1}(\mathfrak m_x)).$$

By Proposition \ref{mld=min},  there exists  $l\in \mathbb N$ such that for any closed point $x\in X$,
$$\mathrm{mld}(x;X,\mathfrak a^{c})=\min_{0\le m\le l}\Big\{\mathrm{codim}(\mathrm{Cont}^{\ge m}(\mathfrak a)\cap\mathrm{Cont}^{\ge 1}(\mathfrak m_x))-cm\Big\}.$$
Since the map $|X| \to \mathbb R_{\ge 0}$, $x\mapsto \mathrm{codim}(\mathrm{Cont}^{\ge m}(\mathfrak a)\cap\mathrm{Cont}^{\ge 1}(\mathfrak m_x))-cm$ is  lower semicontinuous by Lemma \ref{s_m semi conti}, 
the map $|X| \to \mathbb R_{\ge 0}$, $x\mapsto \mathrm{mld}(x;X,\mathfrak a^{c})$ is  lower semicontinuous.

\end{proof}

\begin{cor}
Let $X$ be a smooth variety,  $\mathfrak a\subset \mathcal O_X$ be a non-zero ideal sheaf and  $c\in\mathbb R_{\ge 0}$.
Suppose that $c<\mathrm{lct}(\mathfrak a)$.
Then the set $\{\mathrm{mld}(\eta;X,\mathfrak a^c)\mid \eta\in X\}$ is a finite set.
\end{cor}

\begin{proof}
By Proposition \ref{mld=min},
 there exists  $l\in\mathbb N$ such that 
for any  point $\eta\in X$,
$$\mathrm{mld}(\eta;X,\mathfrak a^{c})=\min_{0\le m\le l}\Big\{\mathrm{codim}(\mathrm{Cont}^{\ge m}(\mathfrak a)\cap\pi_\infty^{-1}(\eta))-cm\Big\}.$$
Assume that $\{\mathrm{mld}(\eta;X,\mathfrak a^c)\mid \eta\in X\}$ is an infinite set.
Then there exist $n\in\mathbb Z_{\ge 0}$ with $0\le n\le l$ and  points $\{\eta_i\}_{i\in\mathbb N}$ of $X$ such that 
$$\mathrm{mld}(\eta_i;X,\mathfrak a^{c})=\mathrm{codim}(\mathrm{Cont}^{\ge n}(\mathfrak a)\cap\pi_\infty^{-1}(\eta_i))-cn\ \mbox{for}\ \ i \in\mathbb N \ \ \mbox{and}$$
$$\{\mathrm{codim}(\mathrm{Cont}^{\ge n}(\mathfrak a)\cap\pi_\infty^{-1}(\eta_i))\mid i\in\mathbb N\}\ \ \mbox{is an  infinite set}.$$
Since $\mathrm{codim}(\mathrm{Cont}^{\ge n}(\mathfrak a)\cap\pi_\infty^{-1}(\eta_i))$ is a non-negative integer, 
we have
$$\sup_{i}\mathrm{codim}(\mathrm{Cont}^{\ge n}(\mathfrak a)\cap\pi_\infty^{-1}(\eta_i))=\infty.$$
This implies that $$\sup_{i}\mathrm{mld}(\eta_i;X,\mathfrak a^c)=\infty,$$
which is a contradiction to $\mathrm{mld}(\eta_i;X,\mathfrak a^c)\le \mathrm{dim} X$  (See \cite[Corollary 3.27]{IR2}).
\end{proof}

Comparing the minimal log discrepancies of a smooth variety with different centers,
we have the following relation,
which is  a standard application of Theorem \ref{mld contact in IR}.
In \cite{A}, Ambro conjectured  the statement holds even for a singular variety.

For a closed subvareity $W$ of a variety $X$, we denote by $\eta_W$ the generic point of $W$.

\begin{prop}[{\cite[Corollary 3.27]{IR2}}]\label{mldV=mldW}
Let $X$ be a smooth variety,  $\mathfrak a\subset \mathcal O_X$ be a non-zero ideal sheaf and  $c\in\mathbb R_{\ge 0}$.
Let $V\subset W$ be  two  proper closed subvarieties of $X$.
Then 
$$\mathrm{mld}(\eta_V;X,\mathfrak a^c)\le\mathrm{mld}(\eta_W;X,\mathfrak a^c)+\mathrm{codim}(V,W).$$
and the equality holds  if  $V$ is very general  in $W$; i.e., $\eta_V$ is in the complement of a countable number of closed subsets in $W$.
Moreover if $\mathrm{char} k=0$, then the equality holds for general $V$ in $W$.
\end{prop}

\begin{rem}
In \cite{IR2}, the above theorem is stated when $c=1$.
However, the same proof works when $c\in\mathbb R_{\ge 0}$. 
\end{rem}

The following shows that the equality in Proposition \ref{mldV=mldW} holds for general $V$
even in positive characteristic, if we assume $c<\mathrm{lct}(\mathfrak a)$.

\begin{prop}\label{open mld}
Let $X$ be a smooth variety, $W$ be a proper closed subvariety of $X$,
 $\mathfrak a\subset \mathcal O_X$ be a non-zero ideal sheaf and  $c\in\mathbb R_{\ge 0}$.
Suppose that $c<\mathrm{lct}(\mathfrak a)$.
Then 
$$\mathrm{mld}(\eta_V; X,\mathfrak a^c)=\mathrm{mld}(\eta_W;X,\mathfrak a^c)+\mathrm{codim}(V,W)$$
for general closed subvariety $V\subset W$.

\end{prop}

\begin{proof}

Let $Z$ be the closed subscheme of $X$ defined by $\mathfrak a$.
First, we assume that $W\not\subset Z$.
Let $U=X\setminus Z$.
Then 
$$\mathrm{mld}(\eta_V; X,\mathfrak a^c)=\mathrm{mld}(\eta_V;X,\mathcal O_X)=\mathrm{dim}X-\mathrm{dim}V$$
for any closed subvariety $V\subset W\cap U$ (See \cite[Corollary 3.27]{IR2}).
Therefore this proposition holds if $W\not\subset Z$.

Next we assume that  $W\subset Z$.
Note that  $\mathrm{ord}_E(\mathfrak a)\ge 1$ for $E$ computing  $\mathrm{mld}(\eta_V;X,\mathfrak a^{c})$ for any closed subvariety $V\subset W$.
By Theorem \ref{mld contact in IR},  Lemma \ref{mld=s_l}   and Lemma \ref{global upper bound lemma},
 there exists $l\in \mathbb N$ such that 
for any closed subvareity $V\subset W$,
$$\mathrm{mld}(\eta_V;X,\mathfrak a^{c})=\min_{1\le m\le l}\Big\{\mathrm{codim}(\mathrm{Cont}^{\ge m}(\mathfrak a)\cap\pi_\infty^{-1}(\eta_V))-cm\Big\}.$$
Let $i:W\to X$ be the closed immersion and $m\in \mathbb N$.
Note  that $\mathfrak a\subset I_W$, where $I_W$ is the defining ideal of $W$.
Since $\mathrm{Cont}^{\ge m}(I_W)_{m-1}\subset\mathrm{Cont}^{\ge m}(\mathfrak a)_{m-1}\cap\mathrm{Cont}^{\ge 1}(I_W)_{m-1}$ and $\pi_{m-1}\big(\mathrm{Cont}^{\ge m}(I_W)_{m-1}\big)=\pi_{m-1}\big(\mathrm{Cont}^{\ge 1}(I_W)_{m-1}\big)=W$,
we have  $$\pi_{m-1}\big(\mathrm{Cont}^{\ge m}(\mathfrak a)_{m-1}\cap\mathrm{Cont}^{\ge 1}(I_W)_{m-1}\big)=W.$$
Thus there is  a surjective morphism $\phi_m :\mathrm{Cont}^{\ge m}(\mathfrak a)_{m-1}\cap\mathrm{Cont}^{\ge 1}(I_W)_{m-1}\to W$ such that $\pi_{m-1}|_{\mathrm{Cont}^{\ge m}(\mathfrak a)_{m-1}\cap\mathrm{Cont}^{\ge 1}(I_W)_{m-1}}=i\circ\phi_m$.
By the generic flatness, there exists an open subset $U_{m}$ of $W$ such that 
$\phi_m^{-1}(U_{m})\to U_{m}$ is flat.
Then
$$\mathrm{dim}\phi_m^{-1}(V)=\mathrm{dim}\phi_m^{-1}(U_{m})-\mathrm{codim}(V,W)$$
 for any closed subvariety $V\subset U_{m}$.
Note that  for any closed subvariety $V\subset U_{m},$
$$\mathrm{dim}\phi_m^{-1}(V)=\mathrm{dim}X_{m-1}-\mathrm{codim}\big(\mathrm{Cont}^{\ge m}(\mathfrak a)_{m-1}\cap\pi_{m-1}^{-1}(\eta_V)\big),$$
$$\mathrm{codim}(\mathrm{Cont}^{\ge m}(\mathfrak a)\cap\pi_\infty^{-1}(\eta_V))=\mathrm{codim}(\mathrm{Cont}^{\ge m}(\mathfrak a)_{m-1}\cap\pi_{m-1}^{-1}(\eta_V)).$$
Let $U'_m$ be a open subset of $X$ such that $U'_{m}\cap W=U_{m}$ and $U=\cap_{1\le m\le l}U'_m$.
Then $U$ is an open subset  of $X$ such that $U\cap W\neq \emptyset$ and for any closed subvariety $V\subset W\cap U$,
$$\mathrm{mld}(\eta_V; X,\mathfrak a^c)=\mathrm{mld}(\eta_W;X,\mathfrak a^c)+\mathrm{codim}(V,W).$$

\end{proof}


\end{document}